\documentclass{amsart}
\usepackage{amssymb}

\newcommand{\Ps}{\mathbf{P}}

\newcommand{\C}{\mathbf{C}}
\newcommand{\Q}{\mathbf{Q}}
\newcommand{\Z}{\mathbf{Z}}
\newcommand{\F}{\mathbf{F}}

\newcommand{\G}{\mathbf{G}}

\newtheorem{lemma}{Lemma}[section]
\newtheorem{proposition}[lemma]{Proposition}
\newtheorem{theorem}[lemma]{Theorem}

\theoremstyle{definition}
\newtheorem{notation}[lemma]{Notation}
\newtheorem{definition}[lemma]{Definition}
\newtheorem{construction}[lemma]{Construction}

\newtheorem{conv}[lemma]{Convention}
\theoremstyle{remark}
\newtheorem{remark}[lemma]{Remark}

\DeclareMathOperator{\rank}{rank}

\DeclareMathOperator{\Div}{Div}

\DeclareMathOperator{\Gr}{Gr}
\DeclareMathOperator{\Hom}{Hom}
\DeclareMathOperator{\NS}{NS}
\DeclareMathOperator{\cha}{char}
\DeclareMathOperator{\MW}{MW}

\DeclareMathOperator{\sing}{sing}

\title{Supersingular  elliptic surfaces and Infinitesimal Torelli}
\author{Remke Kloosterman}
\begin{document}
\begin{abstract}
In 1981 Katsura presented a classification of non-rational Jacobian elliptic surfaces which admit a base change which is rational and therefore are supersingular. In 2004 we presented a classification of Jacobian regular elliptic surfaces which do not satisfy infinitesimal Torelli. These classifications of quite different properties turn out to be  very similar. 

Both of the classifications contain precisely  five positive-dimensional families. Each of these families is of product quotient-type i.e., the quotient of  $C_t\times E$ by a finite group $G$, where $E$ is a fixed elliptic curve and $(C_t)_{t\in T}$ is a family of curves admitting a generalised  Prym variety of dimension one.
In this paper we revisit these families. We give a geometric characterisation of these families, from which it follows directly that the Torelli property fails in characteristic zero and the families are Artin supersingular in positive characteristic.
\end{abstract}

\thanks{The author would like to thank Matthias Sch\"utt and Orsola Tommasi 
for several comments on a previous version of this paper.
The author is a member  of INdAM-GNSAGA. This work was supported  by the UNIPD BIRD-SID-2024 
project ``Moduli spaces and positivity".}
\maketitle
\section{Introduction}
Recently, Church \cite{Church} presented a number of counterexamples to Shioda's conjecture on the unirationality of supersingular surfaces with trivial \'etale fundamental group. Church's counterexamples are of product-quotient type, i.e., the quotient of a product of two curves by a finite group, and these counterexamples are of general type.

Church's method relies on the existence of diagonal symmetric differential forms. It seems unclear whether this method can be extended to construct counterexample of Kodaira dimension 1. 
In a search for interesting product-quotient elliptic surfaces we noticed the paper \cite{KatUni} by Katsura, in which he classified all Jacobian elliptic surfaces  which are not rational but have a base change which is a  rational surface (we will refer to this as BCU). 
One can easily show that the $j$-invariant of such a surface is constant if $p\geq 3$.

The list of Weierstrass equations in \cite{KatUni}  is  very similar to the list in \cite[Theorem 4.8]{Ext} of Weierstrass equations for complex elliptic surfaces over $\Ps^1$ with a section and $p_g>0$ for which the infinitesimal Torelli property fails, the main difference between the two lists being the inclusion of $K3$ examples in Katsura's list and the fact that the authors works over different fields $K$. We list the examples from \cite{KatUni} and \cite{Ext} in Figure~\ref{tblj0} and~\ref{tblj1728}.
\begin{figure}[hbtp] 
\[
\begin{array}{|c|c|c|c|}
\hline
y^2=x^3+f(s,t)& n&\mbox{Failure of } & \mbox{BCU if}    \\
\mbox{with } f&& \mbox{infinitesimal Torelli} & p\equiv 5\bmod 6\\ \hline
t^5(t-s)^5s^2 &2&-& X\\
t^5(t-s)^4s^3 &2&-& X\\
t^4(t-s)^4s^4 &2&-& X\\
t^5(t-s)^5s^5(t-\alpha s)^3&3 &X& X\\
t^5(t-s)^5s^4(t-\alpha s)^4&3&X& X\\
t^5(t-s)^5s^5(t-\alpha s)^5(t-\beta s)^4&4&X& X\\
t^5(t-s)^5s^5(t-\alpha s)^5(t-\beta s)^5(t-\gamma s)^5&5&X&X \\
\hline
\end{array}\]
\caption{$j=0$ examples}\label{tblj0}
\end{figure}
\begin{figure}[hbtp]
\[
\begin{array}{|c|c|c|c|}
\hline
y^2=x^3+f(s,t)x&n& \mbox{Failure of } & \mbox{BCU if}    \\
\mbox{with } f&& \mbox{infinitesimal Torelli} & p\equiv 3\bmod 4\\ \hline
t^3(t-s)^3s^2 &2&-&X\\
t^3(t-s)^3s^3(t-\alpha s)^3&3 &X&X\\
\hline
\end{array}\]
\caption{$j=1728$ examples}\label{tblj1728}
\end{figure}

The Weierstrass equation $y^2=x^3+f$, for the $f$ from the first table, and $y^2=x^3+fx$, for the $f$ from the second table, define each a hypersurface of degree $6n$ in $\Ps(1,1,2n,3n)$. The actual surface for which BCU or the failure of infinitesimal Torelli holds is the minimal resolution of singularities of this hypersurface.
For the failure of infinitesimal Torelli to occur, one picks $\alpha,\beta,\gamma \in \C\setminus\{0,1\}$ pairwise distinct, for BCU one picks $\alpha,\beta,\gamma\in K\setminus \{0,1\}$ pairwise distinct in an algebraically closed field $K$  of characteristic $p$, where $p$ satisfies the mentioned congruence. 

Suppose for the moment that $j=0$. Let $f\in K[s,t]$ be a homogeneous polynomial of degree $6n$ such that each irreducible factor has multiplicity at most 5. Let $g\in K[s,t]$ be a polynomial of minimal degree such that $gf$ is a sixth power. Let $\tilde{C}_f$ be the normalisation of $y^6=f$.

Let $\tilde{X}_f$ be the minimal desingularisation of $y^2=x^3+f$ and define $\tilde{X}_g$ similarly. Then both $\tilde{X}_f$ and $\tilde{X}_g$ are birational to a quotient of $\tilde{C}_f\times E$ by an automorphism of order 6, however, with different actions. 

We will show that if $K=\C$ and $p_g(\tilde{X}_f)>1$ then $\tilde{X}_f$ does not satisfy infinitesimal Torelli if and only if $\tilde{X}_g$ is a rational surface (Theorem~\ref{thmInf}).
Similarly, we will show that if $K$ is an algebraic closed field of characteristic $p$, such that $p\equiv 5 \bmod 6$ and $\tilde{X}_g$ is rational then $\tilde{X}_f$ is supersingular (Theorem~\ref{thmSS}).
Hence both the failure of infinitesimal Torelli as well as the supersingularity aspect of BCU turn out  to be direct consequences of the rationality of $\tilde{X}_g$.

Similar results can be shown for the case $j=1728$. For elliptic surfaces with constant $j$-invariant, different from 0, 1728 the  above-mentioned twisting construction of automorphisms is impossible.

The result in characteristic zero has an interesting application. The curve $\tilde{C}_f$ has an automorphism $\sigma$ of order 6. Denote by $\tilde{C}_f^{(3)}$ the quotient of $C_f$ by $\sigma^2$ and  by $\tilde{C}_f^{(2)}$ the quotient of $\tilde{C}_f$ by $\sigma^3$.
Then the abelian variety $A=J(\tilde{C}_f)/J(\tilde{C}_f^{(2)}+J(\tilde{C}_f^{(3)})$ has dimension $k-3$, but depends only on the choice of $k$ and the various orders of the zeroes of $f$, hence we have a $k-3$-dimensional family of curves with an automorphism of order 6, such that  the associated generalized Prym variety is constant. The result in positive characteristic implies that $A$ is isogeneous to $k-3$ copies of the $j=0$ elliptic curve over an algebraically closed field of characteristic $p\equiv 5 \bmod 6$.

In Section~\ref{secPQ} we describe the product-quotient structure of an elliptic surface with constant $j$-invariant $0$ and describe its cohomology in terms of the cohomology of the curves used to construct the quotient. In Section~\ref{secPF} we use this description to conclude that all surfaces with $j$-invariant 0 in Katsura's list are Artin supsersingular and all examples in our list from \cite[Theorem 4.8]{Ext} do not satisfy infinitesimal Torelli. In Section~\ref{secRMK} we discuss a difference between the Mordell-Weil groups in characteristic zero and positive characteristic and show that a certain Prym-type construction does not satisfy Torelli-type properties. In Section~\ref{secj1728} we discuss the case of $j$-invariant 1728.

\section{Cohomology of elliptic surfaces with constant $j$-invariant}\label{secPQ}
In this section we study the product-quotient structure of an elliptic surface with constant $j$-invariant. Most of the results mentioned in this section are  well-known to the experts, but we include them for the reader's convenience. In this and the next section we consider only the $j=0$ case, the case $j=1728$ is very similar and we list the necessary changes in Section~\ref{secj1728} .

\begin{notation} Fix the following
\begin{itemize}
\item a field $K$, either the complex numbers or an algebraically closed field of characteristic $p$, with $p>3$;
\item  a positive  integer $n$;  
\item a fixed primitive $6$-th root of unity $\zeta\in K$;
\item a homogeneous polynomial $f\in K[s,t]_{6n}$ of degree $6n$, such that every irreducible factor of  $f$ has multiplicity at most 5;
\item the number of distinct zeroes $k$ of $f$ in $K$;
\item  a polynomial $h\in K[s,t]_k$ such that $( h ) =\sqrt{( f)}$; 
\item the polynomial $g=h^6/f\in K[s,t]_{6(k-n)}$;
\end{itemize}

Moreover, for any integer $m$ and  homogeneous form $F \in K[s,t]_{6m}$ such that every irreducible factor of  $F$ has multiplicity at most 5, we fix
\begin{itemize}
\item the surface $X_F=V(-y^2+x^3+F) \subset \Ps(1,1,2m,3m)$ (with coordinates $s,t,x,y$);
\item  the  elliptic surface  $\varphi_F :\tilde{X}_F\to \Ps^1$, associated with $X_F$ (cf. Remark~\ref{rmkWei});
\item the curve $C_F=V(-u^6+F) \subset \Ps(1,1,m)$  (with coordinates $s,t,u$) and its normalization $\tilde{C}_F$;
\item the curve $E=V(-y^2z+x^3+z^3)\subset \Ps^2$;
\item  the point $O=(0:1:0)\in E$;
\item the automorphisms $\tau_F:C_F\to C_F$  defined by $(s:t:u)\mapsto (s:t:\zeta u)$;
 \item the automorphisms $\tilde{\tau_F}:\tilde{C}_F\to\tilde{C}_F$, the lift of $\tau_f$;
\item  the automorphism $\sigma:E\to E$ given by $(x:y:z)\mapsto (\zeta^2x:\zeta^3y:z)$;
\item the automorphism $\rho_F=(\tilde{\tau_F},\sigma):\tilde{C}_F\times E\to \tilde{C}_F\times E$;
\item the rational map $\pi':C_F\times E\dashrightarrow X_F$ sending $((s:t:u),(x:y:z))$ to $(zs:zt:u^2z^{2m-1}x:u^3z^{3m-1}y)$;
\item the rational map  $\pi:\tilde{C}_F\times E \dashrightarrow X_F$, the composition of the morphism $\tilde{C}_F\times E\to C_F\times E$ with $\pi'$.
\end{itemize}
\end{notation}
Note that $h$ is only determined up to a constant and so is  $g$. However, $X_g$ and $C_g$ are uniquely determined up to isomorphism.

\begin{conv}\label{ConvGen}
Let $Y/K$ be a projective variety, possibly singular. 
If $K=\C$ then $H^i(Y)$ stands for the singular cohomology group $H^i(Y,\Q)$ together with its natural mixed Hodge structure (MHS) and we set $L=\Q$. If $K$ is an algebraically closed field of positive characteristic then $H^i(Y)$ stands for rigid cohomology, considered as a Frobenius module, and we set $L$ to be the quotient field of the ring of Witt vectors from $K$.
The choice for rigid cohomology is motivated by the fact that we work often with varieties with isolated singularities.

Both structures come with a weight filtration. For the MHS case see for example \cite{PSbook}, for the rigid cohomology case see \cite{Nak}.
Unless stated otherwise, any map between cohomology groups is understood to be a morphism of MHS if $K=\C$ and a morphisms of Frobenius modules if $\cha(k)=p>0$.  It is well-known  that any such  morphism is a strict morphism of filtered vector spaces.
In almost all cases it follows from the geometric origin of the maps between cohomology groups that they respect the MHS, respectively, the Frobenius module-structure. The single case in the sequel where this is not the case wlll be mentioned explicitly.

Since $Y$ is projective we have that  the highest weight of the weight filtration on $H^i$ equals $i$.
\end{conv}

\begin{remark}\label{rmkWei} Consider the weighted projective space $\Ps(1,1,2n,3n)$ with coordinates $s,t,x,y$.
Let $\Pi=V(s,t)\subset \Ps(1,1,2n,3n)$. If $n>1$ then $\Pi=\Ps(1,1,2n,3n)_{\sing}$.

Let $P\in K[s,t,x,y]$ be a weighted homogeneous polynomial of degree $6n$.
Let $Y=V(P)\subset \Ps(1,1,2n,3n)$. If at least one of the coefficients of $y^2$ and $x^3$ in $P$ is nonzero then $Y\cap \Pi$ consists of a single point $p$.
 
Suppose now that $Y\cap \Pi$ consists of a single point $p$. Consider the projection $\varphi:Y\setminus \{p\}\to \Ps^1$ from $p$ on $V(x,y)$, the $s,t$-line. The general fiber of $\varphi$ is a Weierstrass equation of an affine piece of   an irreducible arithmetic genus one curve.
 By taking the fiberwise closure, i.e., by blowing up the point $p$ on $Y$,  we obtain a family $Y'$ of plane cubics  in a $\Ps^2$-bundle. 
Suppose now that the generic fiber is a smooth cubic. It is then easy to see that $Y'$ has at most finitely many singular points.
A minimal resolution of singularities of $Y'$ yields a smooth projective surface $\tilde{Y}$ together with a morphism $\tilde{\varphi}:\tilde{Y}\to \Ps^1$, which is flat, with generic fiber of  genus one. Moreover, the strict transform on $\tilde{Y}$ of the exceptional divisor   of the first blow-up $Y'\to Y$  is the image of a section of $\tilde{\varphi}$.

It is easy to check that $X_f$ and $X_g$ satisfy the  above imposed genericity constraints.
\end{remark}

The following lemma is immediate:
\begin{lemma}\label{lembir} The rational map $\psi: C_f\dashrightarrow C_g$ defined by $(s:t:u)\mapsto  (s:t:h/u)$ is birational and such that $\psi\circ \tau_f=\tau_g^{-1}\circ \psi$. In particular, $\tilde{C}_f\cong \tilde{C}_g$.
\end{lemma}


\begin{lemma}\label{lemquot} The rational map $\pi':C_f\times E\dashrightarrow X_f$ is regular on $\tilde{C}_f\times (E\setminus \{O\})$.

If $(s_0:t_0)$ is such that $f(s_0:t_0)=0$ then $\pi'(\{(s_0:t_0:0)\}\times E\setminus \{O\})=\{(0:0:s_0:t_0)\}$, i.e., $\pi'$ contracts the fiber of $(s_0:t_0)$.

Let $(x_0:y_0:s_0:t_0)\in X_f$ be a point such that $f(s_0,t_0)\neq 0$ then $(\pi')^{-1}(x_0:y_0:s_0:t_0)$ is finite and consists of at most six points.

The map $\pi'$ is generically finite of degree $6$ and Galois, with Galois group  generated by $\rho_f$.
\end{lemma}
\begin{proof}
If $\pi'$ is not defined at a point $p=((s:t:u),(x:y:z))$ then $zs=zt=0$. Since $C_f$ has no points with $s=t=0$, we obtain that if $z\neq 0$ then $\pi'$ is defined at $p$.  The only point on $E$ with $z=0$ is $O=(0:1:0)$. Hence $\pi'$ is not defined on $C_f\times \{O\}$, and defined everywhere else.

If a point $p_0:=((s_0:t_0:u_0),(x_0:y_0:1))$ belongs to $ C_f\times (E\setminus \{O\})$ then $\pi'(x_0)=(s_0:t_0:u_0^2x_0:u_0^3y)$. Substituting $u_0=0$ yields the second statement.

Suppose now that $f(s_0,t_0)\neq 0$ and $p_1:=((s_1:t_1:u_1),(x_1:y_1:1))$ is a point such that $\pi'(p_0)=\pi'(p_1)$.  Comparing the first two coordinates yields $(s_1,t_1)\neq (0,0)$, since $(s_0,t_0)\neq (0,0)$. After scaling, if necessary, we may assume  $s_0=s_1$ and $t_0=t_1$. Then $u_1^6=f(s_1,t_1)=f(s_0,t_0)=u_0^6$. Hence there exists an $\eta$ such that $\eta^6=1$ and $u_1=u_0\eta\neq 0$. Comparing the final two coordinates of $\pi'(p_0)$ and $\pi'(p_1)$ yields
\[ x_0u_0^2=x_1\eta^2u_0^2 \mbox{ and }  y_0u_0^3=y_1\eta^3u_0^3.\]
From $u_0^6=f(s_0,t_0)\neq 0$ we obtain $x_0=x_1\eta^2,y_0=y_1\eta^3$. Hence there are at most six points in the fiber over $p_0$ and if $x_0y_0\neq0$ then there are precisely six points in a fiber. Moreover $\rho_f$ fixes $\pi^{-1}(p_0)$ and $\langle \rho_f\rangle$ acts both faithfully and transitively on $\pi^{-1}(p_0)$ whenever $f(s_0,t_0)\neq 0$.
\end{proof}

\begin{remark}  \label{remquo}
Let $G=\langle \tilde{\rho_f}\rangle$.
From the previous lemma it follows that $(\tilde{C}_f\times E)/G$ is birational to $\tilde{X}_f$. Hence $H^2(\tilde{C}_f\times E)^G$ and $H^2(\tilde{X}_f)$ should contain nontrivial isomorphic substructures. In order to make this more precise we will consider cohomology up to substructures whose Hodge polygon (if $K=\C$) or Newton polygon ($\cha(K)>0$) has  only slope 1, since such substructures do not admit variation of Hodge structures and are automatically supersingular.


By the K\"unneth formula we obtain that $H^2(\tilde{C}_f\times E)^G$ is a direct sum of  $(H^1(\tilde{C}_f)\otimes H^1(E))^G$ and $H^2(E)\otimes H^0(\tilde{C}_f)\oplus H^0(E)\otimes H^2(\tilde{C}_f)$, where the latter factor is generated by classes of divisors. Therefore we will merely focus on $(H^1(\tilde{C}_f)\otimes H^1( E))^G$ instead of $H^2(\tilde{C}_f\times E)^G$.
\end{remark}
\begin{notation}
For a root of unity $\eta$ we denote  with $H^i(\tilde{C}_f)^\eta$ the $\eta$-eigenspace of $\tilde{\tau_f}^*$ and with $H^i(E)^\eta$ the $\eta$ eigenspace of $\sigma^*$.
\end{notation}
\begin{lemma}\label{lemDecomp} We have
 \[ \left(H^1(\tilde{C}_f)\otimes H^1(E)\right)^{\langle \rho_f\rangle}\cong \left(H^1(\tilde{C}_f)^\zeta\otimes H^1(E)^{\zeta^5}\right)\oplus \left(H^1(\tilde{C}_f)^{\zeta^5} \otimes H^1(E)^\zeta\right)\]
 and
\begin{eqnarray*} \left(H^1(\tilde{C}_g)\otimes H^1(E)\right)^{\langle \rho_g\rangle}&\cong& \left(H^1(\tilde{C}_g)^\zeta\otimes H^1(E)^{\zeta^5}\right)\oplus \left(H^1(\tilde{C}_g)^{\zeta^5} \otimes H^1(E)^\zeta\right)
\\ &\cong& \left(H^1(\tilde{C}_f)^{\zeta^5}\otimes H^1(E)^{\zeta^5}\right)\oplus \left(H^1(\tilde{C}_f)^{\zeta} \otimes H^1(E)^{\zeta}\right).\end{eqnarray*}
 \end{lemma}
 
 \begin{proof}  The first and second isomorphism are a consequence of the fact that  the only eigenvalues of $\sigma^*$ on $H^1(E)$ are $\zeta,\zeta^5$.
  The third isomorphism follows from the isomorphism $\tilde{C}_f\to \tilde{C}_g$ from Lemma~\ref{lembir}.
 \end{proof}

Let $Y$ be a smooth projective surface and $C\subset Y$ be a curve, then there is a cycle class  $c(C)\in H^2(Y)$. This induces a cycle class map $c:\Div(Y)\to H^2(Y)$.  By abuse of notation, we denote with $\NS(Y)_L$ the image of $c(\Div(Y))\otimes_{\Z} L$, where $L$ is the field from Convention~\ref{ConvGen}.

\begin{definition} Let $Y_1,Y_2$ be two projective surfaces. We say that 
two sub-$\Q$-mixed Hodge structures (if $K=\C$), respectively, sub-Frobenius modules $V_1\subset H^2(Y_1)$ and $V_2\subset H^2(Y_2)$  \emph{differ by a pure slope one-structure} if there exists a smooth projective surface $Z$ and birational morphisms $\psi_i:Z\to Y_i$, and a substructure $V\subset H^2(Z)$ such that $\psi_i^*\Gr_2^W H^2(Y_i)+V=H^2(Z)$ for $i=1,2$ and if $K=\C$ then $V$ is pure $\Q$-Hodge structure of type $(1,1)$ and if $\cha(K)=p>0$ then $V$ is a Frobenius module such that its Newton polygon  of $V$ has only slope 1.
\end{definition}

\begin{remark} By Leftschetz-(1,1) one has that the vector space $V$ of the above definition is generated by classes of divisors if $K=\C$. In positive characteristic, one has that every structure generated by classes of divisors is of pure weight 1, but the converse is only true  assuming the Tate conjecture.

Since $W_1H^2(Y_i)$ is in the kernel of $\psi_i^*$ one can equivalently take substructures of $\Gr_2^W H^2(Y_i)$ and consider their preimages under the projection map $H^2(Y_i)\to \Gr_2^W H^2(Y_i)$.
\end{remark}

\begin{lemma} Let $Y_1$ and $Y_2$ be projective surfaces. If $Y_1$ and $Y_2$ are birational then $\Gr_2^W H^2(Y_1)$ and $\Gr_2^W H^2(Y_2)$ differ only by a pure slope one-structure. 
\end{lemma}

\begin{proof} 
We consider first the case where $Y_1$ and $Y_2$ are related by a birational morphism $\psi:Y_2\to Y_1$, with $Y_2$ smooth. Let $\Sigma\subset Y_1$ be the smallest closed subset such that $\psi:Y_2\setminus D\to Y_1\setminus \Sigma$ is an isomorphism, where $D=\psi^{-1}(\Sigma)$.
We have then an exact sequence
\[ H^1(D)\to H^2(Y_1)\to H^2(Y_2)\oplus H^2(\Sigma)\to H^2(D)\to H^3(Y_1).\]
Recall that both and  $D$ is proper and of dimension at most one,  therefore $W_1H^1(D)=H^1(D)$. Taking $\Gr^2_W$ yields
\[ 0 \to \Gr_2^W H^2(Y_1)\to \Gr_2^W H^2(Y_2) \oplus \Gr_2^W H^2(\Sigma)\to \Gr_2^W H^2(D)\]
is exact. Moreover, since $Y_2$ is smooth we have that $\Gr_2^W H^2(Y_2)=H^2(Y_2)$. Let $D=D_1\cup\dots\cup D_m$  be a decomposition of $D$ into irreducible components. Since $D$ has dimension one we find $\Gr_2^W H^2(D)=\oplus_{i=1}^m \Gr_2^W H^2(D_i)$ and the latter is obviously of pure slope 1. Similarly $\Sigma$ is proper and of dimension at most one. Therefore $\Gr_W^2 H^2(\Theta)$ is of pure slope one. This finishes the case where $\psi$ is a morphism and $Y_2$ is smooth.

If both $Y_1$ and $Y_2$ are smooth then any birational map can be factored as a sequence of blow-up and blow-down morphisms with center a point. From the above it now follows that $H^2(Y_1)$ and $H^2(Y_2)$ differ by a pure slope one-structure. 

If  either $Y_1$ or $Y_2$ is singular then there exist resolution of singularities $\tilde{Y_1}$ and $\tilde{Y_2}$ of $Y_1$ and $Y_2$ respectively. Then using the first paragraph it follows that $\Gr_2^W H^2(Y_i)$ and $\Gr_2^W H^2(\tilde{Y_i})$ differ by a pure slope one-structure where $i\in \{1,2\}$. The second part implies now that $\Gr_2^W H^2(Y_1)$ and $\Gr_2^W H^2(Y_2)$  differ by a pure slope one-structure.
\end{proof}

\begin{proposition}\label{prpDifference} The groups $H^2(\tilde{X}_f)$ and $(H^1(\tilde{C}_f)\otimes H^1(E))^G$ differ by a pure slope one-structure.
\end{proposition}
\begin{proof}
From Lemma~\ref{lemquot} it follows that $(\tilde{C}_f\times E)/G$ and $\tilde{X}_f$ are birational. Hence $H^2(\tilde{X}_f)=\Gr_2^W H^2(\tilde{X}_f)$ and $\Gr_2^W H^2((\tilde{C}_f\times E)/G)$ differ by a pure slope one-structure. As pointed out in Remark~\ref{remquo} the graded pieces $\Gr_2^WH^2((\tilde{C}_f\times E)/G)$ and $\Gr^W_2 (H^1(\tilde{C}_f)\otimes H^1(E))^G$ differ by a pure slope one-structure. Since both $\tilde{C}_f$ and $E$ are smooth we have   $\Gr_2^W (H^1(\tilde{C}_f)\otimes H^1(E))^G=(H^1(\tilde{C}_f)\otimes H^1(E))^G$.
\end{proof}

\section{Proofs}\label{secPF}
In this section we show that the examples mentioned in \cite[Theorem 4.6]{Ext} do not satisfy infinitesimal Torelli when $K=\C$, and the examples from \cite{KatUni} are Artin supersingular if $K=\F_q$.
\subsection{Rationality}
\begin{lemma}\label{lemRat} We have that  $\tilde{X}_g$ is a rational surface if and only if $f$ is one of the following
\begin{enumerate}
\item $f$ has two zeroes in $K$ with order $5,1$; order $4,2$ or order $3,3$;
\item $f$ has three zeroes in $K$ with order $5,5,2$ or order $5,4,3$ or $4,4,4$;
\item $f$ has four zeroes in $K$ with order $5,5,5,3$ or $5,5,4,4$;
\item $f$ has five zeroes in $K$ with order $5,5,5,5,4$;
\item $f$ has six zeroes in $K$, each of order $5$.
\end{enumerate}
\end{lemma}

\begin{proof}
The surface $\tilde{X}_g$ is rational if and only if $\deg(g)=6$ and $g$ has at least two zeroes. This limits the possibile factorisations to two factors with exponents $(5,1)$,$(4,2)$ or $(3,3)$; three factors with exponents $(4,1,1)$, $(3,2,1)$ or $(2,2,2)$; four factors with exponents $(3,1,1,1)$ or  $(2,2,1,1)$; five factors with exponents $(2,1,1,1,1)$ or six linear factors. The possible factorisation patterns for $f$ now follow immediately.
\end{proof}
\begin{remark}
If $f$ is chosen such that $\tilde{X}_g$ is rational and $f$ has two distinct factors  then $\tilde{X}_f$ is rational; if $f$ has three distinct factors then $X_f$ is $K3$, if $f$ has at least four factors then the Kodaira dimension of $\tilde{X}_f$ equals 1.
\end{remark}

\subsection{Infinitesimal Torelli}

\begin{construction}
Let $\psi:\mathcal{X}\to\mathcal{B}$ be a flat family of smooth complex manifolds of dimension $d$, and assume that $\mathcal{B}$ is contractible. Let $0\leq m \leq k \leq 2d$ be integers.
Fix a point $0\in \mathcal{B}$. Let $b^{m,k}=\dim F^m H^k(X_0,\C)$. Let $b\in \mathcal{B}$ be a point and let $X_b$ be $\psi^{-1}(b)$. From Ehresmann's theorem it follows that $\psi$ is trivial as $C^\infty$ fibration. From this we obtain a natural isomorphism of abelian groups $\xi_b:H^k(X_b,\Z)\stackrel{\sim}{\longrightarrow} H^k(X_0,\Z)$. Tensoring with $\C$ yields a bijective linear map which by abuse of notation we denote by $\xi_b:H^k(X_b,\C)\to H^k(X_0,\C)$. However the latter map is not an isomorphism of MHS. The failure of being an isomorphism  can be used to define  the period map
\[ \mathcal{P}^{m,k}:\mathcal{B}\to \G(b^{m,k},H^k(X_0,\C))\]
sending $b$ to $\xi_b(F^m H^k(X_b,\C))$.
\end{construction}
\begin{remark}
 By Griffiths' transversality the differential of $\mathcal{P}^{m,k}$
\[ d\mathcal{P}^{m,k} : T_{\mathcal{B},b} \to \Hom(F^mH^k(X_b),H^k(X_b)/F^m H^k(X_b))\]
takes values in $\Hom(F^mH^k(X_b),F^{m-1}H^k(X_b)/F^m H^k(X_b))$.
\end{remark}
\begin{definition}
Suppose now that $\psi$ is as before, and, moreover, it  is a family of surfaces. Let  $\Lambda \subset H^2(X_0,\C)$ be a subgroup  generated by classes of divisors. Then we say that $\psi$ is $\Lambda$-polarized if for every $b\in \C$ the subset $\xi_b^{-1}(\Lambda) \subset H^2(X_b,\C)$ is generated by divisors. \end{definition}
\begin{remark}
If $\psi$ is $\Lambda$-polarized then the differential  $d\mathcal{P}^{2,2}$ takes values in
\[\Hom(H^{2,0}(X_b),(\Lambda ^{\perp}\otimes \C \cap F^{1}H^2(X_b))/ H^{2,0}(X_b)).\]
In particular, if $\Lambda\otimes \C$ equals $H^{1,1}(X,\C)$ then  $\Lambda^{\perp}\otimes \C = H^{2,0}\oplus H^{0,2}$ and the differential  $d\mathcal{P}^{2,2}$ is zero.

The differential of the period map $\partial\mathcal{P}^{p,k}(u)$ equals the cup product with $\rho(u)$, where $\rho:T_{\mathcal{B},0} \to H^1(X_0,\Theta_{X_0})$ is the Kodaira-Spencer map. We say that $X_b$ satisfies infinitesimal Torelli if the map $H^1(X_b,\Theta_{X_b})\to \Hom (H^q(\Omega^p_{X_b}), H^{q+1}(\Omega^{p-1}_{X_b}))$ is injective. 
\end{remark}

\begin{proposition}
Suppose $K=\C$.
If $X_g$ is a rational surface then $\NS(\tilde{X}_f)^\perp\subset H^2(\tilde{X}_f)$ has trivial $(1,1)$-part.
\end{proposition}

\begin{proof}
Recall that $\sigma^*$ acts as multiplication by $\zeta^5$ on $H^0(K_E)$. Therefore we obtain $H^0(K_E)=H^{1,0}(E)=H^1(E)^{\zeta^5}$ and $H^{0,1}(E)=H^1(E)^{\zeta}$. Since $X_g$ is rational  we have that $H^2(X_g)$ is of pure $(1,1)$-type. In particular,
\[ \left(H^1(\tilde{C}_g)^{\zeta} \otimes H^1(E)^{\zeta^5}\right) \oplus\left( H^1(\tilde{C}_g)^{\zeta^5}\otimes H^1(E)^\zeta\right)\]
is of pure $(1,1)$-type. Therefore $H^1(C_g)^\zeta$ is of pure $(0,1)$-type and $H^1(C_g)^{\zeta^5}$ is of pure $(1,0)$-type. Since $H^1(\tilde{C}_g)^\zeta\cong H^1(\tilde{C}_f)^{\zeta^5}$  by Lemma~\ref{lembir} we find that $(H^1(C_f)\otimes H^1(E))^G$ has trivial $(1,1)$-part. Since this equals $H^2(\tilde{X}_f)$ up to divisors, it follows that $\NS(\tilde{X}_f)^{\perp}$ has trivial  $(1,1)$-part.
\end{proof}

\begin{theorem}\label{thmInf}
Suppose $K=\C$. If $\tilde{X}_g$ is rational and $k>3$, then $p_g(\tilde{X}_f)>1$ and $\tilde{X}_f$ does not satisfy infinitesimal Torelli.
\end{theorem}
\begin{proof}
The family of all $\tilde{X}_f$ with the same factorization type up to automorphisms has dimension $k-3$. Moreover, the degree of the minimal discriminant of $\tilde{X}_f$ equals $12(k-1)$. Hence $p_g(X)=k-2>1$.

Let $\Lambda$ be the generic N\'eron-Severi group. Then the $\tilde{X}_f$ form a $\Lambda$-polarized family of smooth
surfaces. The previous proposition shows that $h^{1,1}(\Lambda^{\perp})=0$. This implies that the differential of the period map, when restricted to the tangent space of this family, takes values in $\Hom (H^{(2,0)},0)$, hence this differential is zero.

Since the image of the tangent space to this family in  $H^1(X,\Theta_X)$ under the  Kodaira-Spencer map is nonzero it follows that the infinitesimal Torelli property does not hold.\end{proof}

\begin{remark}
In \cite{Ext} we proved the above theorem by identifiying $h^{1,1}$ linearly independent  divisors classes, instead of exploiting the product-quotient structure. 
\end{remark}

\subsection{Artin supersingularity}
We switch now to the positive characteristic side:

\begin{definition}
If $\cha(K)=p>0$ then we say that a smooth projective variety $Y/K$ is supsersingular if  for each $j\in \{0,1,\dots,2\dim(Y)\}$ all slopes of the Newton polygon of the crystalline cohomology $H^j(Y,W(K))$ are equal to $j/2$.
\end{definition}

For a variety $Y$ over a finite field with $q$ elements, this is equivalent to saying that the eigenvalues of Frobenius on the $\ell$-adic cohomology $H^j(X,\Q_\ell)$ are equal to the product of $q^{j/2}$ with  roots of unity.

We defined supersingular only for smooth projective varieties. For historical reasons we used crystalline cohomology in this definition. However, for a smooth projective variety  there is a comparison theorem between crystalline cohomology tensored with  $L$ and the rigid cohomology. In particular, it suffices to consider the Newton polygon of  the rigid cohomology groups.

A non-trivial elliptic surface over $\Ps^1$ satisfies $h^0=h^4=1$ and $h^1=h^3=0$. In order to prove that such a surface is supersingular it suffices to determine the slopes of Frobenius  on $H^2$.

\begin{theorem}\label{thmSS}
If $p\equiv 5 \bmod 6$ and $\tilde{X}_g$ is rational then $\tilde{X}_f$ is Artin supersingular.
\end{theorem}
\begin{proof}
Since $p\equiv 5 \bmod 6$ we have that all slopes of Frobenius on $H^1(E)$ are $\frac{1}{2}$.
As $\tilde{X}_g$ is rational we have that all slopes on $H^2(\tilde{X}_g)$ are $1$, thus the same holds for the substructure   $\left(H^1(\tilde{C}_g)\otimes H^1(E)\right)^{\langle \rho_g\rangle}$, which is isomorphic to
\[ \left( H^1(\tilde{C}_f)^\zeta\otimes  H^1(E)^\zeta\right) \oplus\left(H^1(\tilde{C}_f)^{\zeta^5}\otimes H^i(E)^{\zeta^5}\right).\]
by Lemma~ \ref{lemDecomp}. Therefore  all slopes of Frobenius on $H^1(\tilde{C}_g)^{\zeta}\cong H^1(\tilde{C}_f)^{\zeta^5}$ and $H^1(\tilde{C}_g)^{\zeta}\cong H^1(\tilde{C}_g)^\zeta$ are $\frac{1}{2}$. Hence Frobenius has only slopes $1$ on 
\[ (H^1(\tilde{C}_f)\otimes H^1(E))^{\langle \rho_f \rangle}\cong H^1(\tilde{C}_f)^\zeta\otimes  H^1(E)^{\zeta^5} \oplus H^1(\tilde{C}_f)^{\zeta^5}\otimes H^i(E)^\zeta.\]
 Since this latter cohomology group differs from  $H^2(\tilde{X}_f)$ only by substructure of only slope one (Proposition~\ref{prpDifference}), we obtain that $X_f$ is Artin supersingular.
\end{proof}

\section{Two observations}\label{secRMK}

We keep the notation of the previous two sections. In particular, we have $n=\deg(f)/6$ and that $k$ equals the number of zeroes of $f$. Moreover we will assume that $X_g$ is rational, which is equivalent to assuming $k=n+1$.

Suppose $K=\C$ for the moment.
For $a\in \{2,3,6\}$ let $C_f^{(a)}$ be the curve $u^a=f(s,t)$ in $\Ps(1,1,(6/a)n)$.
Let $D_n^{(a)}\subset  \Ps(1,1,(6/a)n)$ be the curve $u^a=s^{6n}+t^{6n}$.

\begin{lemma} We have $\chi(D_n^{(a)})=2a-(1-a)6n$.\end{lemma}
\begin{proof} Apply Riemann--Hurwitz formula to the $a$-fold cover $D_n^{(a)}\to\Ps^1$.
\end{proof}

\begin{lemma} We have $\chi(C_f^{(a)})=(1-a)6n+2a+(a-1)(6n-k)=2a-k(a-1)$. 
\end{lemma}
\begin{proof}
From \cite[Corollary 5.4.4]{Dim} it follows that $\chi(C_f^{(a)})=\chi(D_n^{(a)})+\mu_f^{(a)}$, where $\mu_f^{(a)}$ is the total Milnor number of $C_f^{(a)}$.

Let $R\in C_f^{(a)}$ be a singular point. Then locally $C_f^{(a)}$ has a local equation $z^a=w^b$ in a neighbourhood of $R$, where $b=b(p)$ satisfies $2\leq b \leq 5$. The Milnor algebra of such singularity equals $\C[z,w]/(z^{a-1},w^{b-1})$ which is a $\C$-vector space of dimension $(a-1)(b-1)$. 

Now the $\sum_{p\in C_f^{\sing}} (b(p)-1)=\deg(f)-\deg(h)=6n-k$. Therefore $\mu_f^{(a)}=(a-1)(6n-k)$.
\end{proof}

\begin{proposition} We have  $\dim H^1({C_f})^\zeta=k-2$.
\end{proposition}

\begin{proof}
From the previous lemmata it follows that
\[ \sum_{i=0}^2(-1)^i \dim (H^i(C_f^{6})^\zeta \oplus H^i(C_f^{6})^{\zeta^5})\]
equals
\[\chi(C_f^{(6)})-( (\chi(C_f^{(2)})-\chi(\Ps^1))+(\chi(C_f^{(3)})-\chi(\Ps^1))+\chi(\Ps^1)).\]
which in turn equals $\chi(C_f^{(6)})-\chi(C_f^{(2)})-\chi(C_f^{(3)})+2$.

From the previous lemma it follows now that 
\[ \chi(C_f^{(6)})-\chi(C_f^{(2)})-\chi(C_f^{(3)})+2=2(6-2-3+1)-k(5-1-2)=4-2k  \]
and therefore $H^1(C_f)^{\zeta}$ has dimension $k-2$.
\end{proof}

We now drop the assumption $K=\C$, pass to the resolution of singularities.
\begin{proposition} We have that $\dim H^1(\tilde{C}_f)^\zeta=k-2$.
\end{proposition}
\begin{proof}
If $K=\C$ then let $R\in C_f^{(a)}$ be a singular point. Then  $C_f^{(a)}$ has a local equation of the form $z^a=w^b$ in a neighbourhood of $R$, where $b=b(p)$ satisfies $2\leq b \leq 5$. The Milnor algebra of such  a singularity equals $\C[z,w]/(z^{a-1},w^{b-1})$, which is a $\C$-vector space of dimension $(a-1)(b-1)$. The dimension of the local cohomology $H^1_p(C_f^{(a)})$ equals the dimension of the space of elements in the Milnor algebra of weighted degree $1-\frac{1}{a}-\frac{1}{b}$, with $\deg(z)=\frac{1}{a}$ and $\deg{w}=\frac{1}{b}$. Since this dimension for $a=6$ equals the sum of the dimension for $a=2$ and $a=3$ we find that the only eigenvalues of $\tau_f$ on these groups are $-1$ and third roots of unity. In particular $H^1(\tilde{C}_f)^\zeta=H^1(C_f)^\zeta$.

If $K\neq \C$ then $K$ is a field of characteristic at least 5. Then we can lift $\tilde{C}_f$ to characteristic zero and obtain the result from the characteristic zero result. 
\end{proof}

\begin{proposition} The abelian variety $J(\tilde{C}_f)/J(\tilde{C}_f^{(2)})+J(\tilde{C}_f^{(3)}))$  has dimension $k-3$ and is locally constant as a function in $f$.
\end{proposition}
\begin{proof}
The dimension statement follows from the previous result.

The fact that this family is locally constant follows directly from Theorem~\ref{thmInf} If $k=\C$, or from  a lift to characteristic zero if $K$ is an algebraically closed field of characteristic at least 5.
\end{proof}

The second observation compares characteristic zero and characteristic $p$. Let $K$ be either $\C$ or an algebraically closed field of characteristic $p>0$.

\begin{lemma} If $K=\C$ and $X_g$ is rational then 
\[ \rank \MW(\varphi_f)=0 \mbox { and } \rank \MW(\varphi_g)=2(n-1).\]
If $K=\overline{\F_p}$ with $p\equiv 5 \bmod 6$ and $X_g$ is rational then 
\[ \rank \MW(\varphi_f)=\rank \MW(\varphi_g)=2(n-1).\]
\end{lemma}
\begin{proof}  Note that the singular fibers of $X_g$ correspond to the zeros of $g$, and that each singular fiber is additive. A rational elliptic surface with $k$ additive singular fibers has geometric Mordell-Weil rank $2(k-2)$. 
This yields the statements for $\varphi_g$.

A non-torsion section of $\varphi_f$ induces a morphism $\tilde{C}_g\to E$. The class of the graph of this morphism yields a nonzero element of $H^1(\tilde{C}_g)\otimes H^1(E)$ of Hodge type $(1,1)$. (One takes the class in $H^2(\tilde{C}_f\times E)$ and then apply the orthogonal projection on $((H^0(\tilde{C}_f)\otimes H^2(E))\oplus (H^2(\tilde{C}_f)\otimes H^0(E)))^\perp$.)

If $K=\C$ then there are no nonzero class in $H^1(\tilde{C}_g)\otimes H^1(E)$ of type $(1,1)$. Hence $\MW(\varphi_f)$ is finite.

However, if $K$ has characteristic  $p\equiv 5 \bmod 6$. Then the pull back of a section of $\varphi_g$ under a Frobenius base change yields a section over $\varphi_{f'}$, where $f'\in K[s,t]$ is a polynomial such that $f'$ and $g^p$ have the same zeroes in $K$, the difference of order of vanishing is a multiple of six, and $f'$ has only factors with multiplicity at most 5.

Hence $\rank \MW(\varphi_f')=\rank\MW(\varphi_g)$. Since $p\equiv -1 \bmod 6$ we obtain that the multiplicity of factors of $f$ and $f'$ agree. Since $K$ is perfect we can find  a $g'$ such that its pull back under Frobenius equals  $f$ modulo sixth powers.
\end{proof}
\begin{remark}
From this result it follows directly that if $k>3$ and $p\equiv -1 \bmod 6$ then  $J(\tilde{C}_f)/J(\tilde{C}_f^{(2)})+J(\tilde{C}_f^{(3)}))$ is isogeneous to a product of $k-3$ copies of a $j=0$ elliptic curve.
\end{remark}

\section{Case $j=1728$}\label{secj1728}
Up to this point we considered only the case $j=0$. The case $j=1728$ can be obtained similarly and we will now discuss the changes.

For $j=1728$  we take $f\in K[s,t]_{4n}$ and $f$ should have factors of multiplicity at most 3. The polynomial $h$ is defined as before, but $g=h^4/f$. For the definition of $X_f$ and $X_g$ one takes $y^2=x^3+fx$ and $y^2=x^3+gx$.
The curves $C_f$ and $C_g$ are defined to be $u^4=f$ and $v^4=g$ respectively and for $E$ one takes the elliptic curve $y^2z=x^3-xz^2$. One considers now primes $p$ such that $p\equiv 3\bmod 4$.
 The root of unity $\zeta$ is replaced by $i$, with $i^2=-1$ and  everywhere we wrote $\zeta^5$ we replace this by $-i$. With these modifications all results stand, except for the classification of possible $f$, which is given in the introduction.

\bibliographystyle{plain}
\bibliography{remke2}

\end{document}